\documentclass[11pt]{article}
\usepackage{hydstokes-sipdg} % My definitions

\usepackage[utf8]{inputenc}
\usepackage{a4wide}
\usepackage[hidelinks]{hyperref}

\usepackage{amsmath}
\usepackage{amssymb}
\usepackage{amsthm}
\newtheorem{theorem}{Theorem}

\newtheorem{lemma}[theorem]{Lemma}
\newtheorem{corollary}[theorem]{Corollary}

\theoremstyle{remark}
\newtheorem{remark}{Remark}
\theoremstyle{definition}

\usepackage{graphicx}

\begin{document}

\title{Stable SIP Discontinuous Galerkin Approximations of the Hydrostatic Stokes Equations}
%\date{}
\author{F. Guill\'en Gonz\'alez, M.V. Redondo Neble, J.R. Rodr\'iguez Galv\'an}
\bibliographystyle{alpha}
\maketitle

\begin{abstract}
  We propose a Discontinuous Galerkin (DG) scheme for the numerical
  solution of the Hydrostatic Stokes equations in Oceanography. This
  new scheme is based on the introduction of the symmetric interior
  penalty (SIP) technique for the Hydrostatic Stokes mixed variational
  formulation. Recent research showed that stability of the mixed
  formulation of Primitive Equations requires LBB
  (Ladyzhenskaya--Babu\v{s}ka--Brezzi) inf-sup condition and an extra
  hydrostatic inf-sup restriction relating the pressure and the
  vertical velocity. This hydrostatic inf-sup condition invalidates
  usual Stokes continuous finite elements like Taylor-Hood $\P2/\P1$
  or bubble $\P{1_b}/\P1$. Here we consider $\P{k}/\P{k}$
  discontinuous finite elements and, using adequate LBB-like and
  hydrostatic discrete inf-sup conditions we can demonstrate stability
  of the SIP DG scheme in the natural energy norm for this
  problem. Finally, according numerical tests are provided.
\end{abstract}

%\noindent {\bf Keywords:}
%Vesicle membranes, Navier-Stokes equations,  Cahn-Hilliard equation, energy dissipation, convergence to equilibrium, Lojasiewicz-Simon's inequalities.

%%%%%%%%%%%%% SECTION %%%%%%%%%%%%%%%
\section{Introduction}  \label{intro}

In this work we delve into the stability of a discrete Discontinuous
Galerkin formulation for the Hydrostatic Stokes equations (or
Primitive Equations of the ocean), where a penalization of interior
jumps of velocity, based on the symmetric interior penalty (SIP)
technique~\cite{arnold_interior_1982}, is introduced. We show that
this new formulation allows writing the equations as a mixed
(Stokes-like) problem which satisfies the well-known LBB
(Ladyzhenskaya--Babu\v{s}ka--Brezzi) condition and also the
hydrostatic inf-sup restriction which has been observed in the
Hydrostatic Stokes
model~\cite{Azerad:00,Azerad:PhD:96,Azerad:1994,Guillen-RRGalvan:stokes:APNUM:16,Guillen-RRGalvan:hydrostatic:NUMAT:15,Guillen-RRGalvan:stabilized:SINUM:15}.
Thus, the Hydrostatic Stokes equations can be approximated in standard
Finite Element (FE) meshes without vertical integration (customary in
most ocean and atmosphere models).

The equations of geophysical fluid dynamics governing the motion of
the ocean and atmosphere are derived from the conservation laws from
physics.  In the case of large scale ocean (see
e.g. \cite{CushmanRoisin-Beckers:09}),
% it is considered as made up of a slightly compressible fluid modelled by
% conservation of momentum and mass equations, with
% variable density (depending on temperature and salinity) and
% Coriolis acceleration.
the resulting system is too complex and, from a practical point of
view, numerous simplifications are
introduced, including the ``small layer'' hypothesis:
\begin{equation*}
  \overline\varepsilon = \frac{\text{vertical scale}}{\text{horizontal scale}} \quad
  \text{ is very small,}
  % \label{eq:vertical-horizontal-epsilon}
\end{equation*}
for example a few Kms over some thousand Kms, that is
$\overline\varepsilon\simeq 10^{-3},10^{-4}$.

Variables like temperature and salinity will not be considered, so
that constant density is assumed (although this work could be extended
to the general variable-density case in future works).  Thus we can
focus on the momentum law, leading to the Navier-Stokes equations. The
anisotropic domain, after a vertical scaling, is transformed into the
following isotropic or adimensional (independent of
$\overline\varepsilon$) domain
$$
\Omega = \bigl\{ (\xx,z)\in \Rset^{3} \ /\ \xx=(x,y)\in\surface,\
-D(\xx)< z < 0 \bigr\},
$$
where $S\subset\Rset^2$ is the surface domain and $D=D(\xx)$ is the bottom function.
Here,  the rigid lid hypothesis has been
assumed (no vertical displacements of the free surface of the
ocean). We decompose the boundary into three parts: the
surface, $\surfaceBoundary =\overline\surface\times\{0\}$, the
bottom,
$\bottomBoundary=\{(\xx,-D(\xx)) \ /\ \xx=(x,y)\in\surface \}$, and
the talus or lateral walls,
$\talusBoundary=\{(\xx,z) \ /\ \xx\in\partial\surface,
-D(\xx)<z<0\}$.

Finally, a $\overline\varepsilon$-dependent scaling of vertical velocity is introduced
% ($\vv
% =\vv^\varepsilon / \varepsilon$, where $\vv^\varepsilon$ is the original
% dimensional variable in $\Omega_\varepsilon$,
(see~\cite{Azerad-Guillen:01}), leading to the following equations
in the time-space domain $(0,T)\times\Omega$
(called \emph{Anisotropic} or \emph{Quasi-Hydrostatic
  Navier-Stokes Equations} and, for the limit case $\overline\varepsilon=0$,
\emph{Hydrostatic Navier-Stokes} or \emph{Primitive Equations}) where we denote
$\varepsilon = \nu\overline\varepsilon^2$:
\begin{align}
  \dt\uu + (\uu\cdot\gradx)\uu +\vv\dz \uu
  - \nu\Delta\uu + \gradx\pp &= \ff,
                               \label{eq:NS.a}
  \\
  \varepsilon\big\{\dt\vv + (\uu\cdot\gradx)\vv +\vv\dz \vv
  - \nu\Delta\vv \big\} + \dz\pp &= - g,
                                   \label{eq:NS.b}
  \\
  \divx\uu +  \dz\vv &= 0.
                       \label{eq:NS.c}
\end{align}
Here $\gradx=(\dx,\dy)^T$, $\divx\uu=\dx u_1+\dy u_2$ and
$\nu$ is the (adimensional kinematic) viscosity. The unknowns are
the 3D velocity field,
$(\uu, \vv):\Omega \times (0,T) \rightarrow \Rset^3$ and the
pressure, $\pp:\Omega \times (0,T) \rightarrow \Rset$. The term
$\ff=(f_1,f_2)^T$ models a given horizontal force while $g$ involves
the force due to gravity, which can be written in a potential form and
incorporated to the pressure term, hence it can be assumed $g=0$ in
(\ref{eq:NS.b}).  Other phenomena like the effects due to
the Coriolis acceleration are not considered  because
they are linear terms not affecting to the results presented in this work.
The system is endowed with initial values for the velocity field,
$(\uu,\vv)|_{t=0} = (\uu_0,\vv_0)$ and adequate boundary conditions,
for instance:
\begin{align}
  \nu\dz\uu|_{\surfaceBoundary} = \gs,
  \quad
  \vv|_{\surfaceBoundary} &=0,
                            \label{eq:bc.1}
  \\
  % \intertext{adherence on bottom and talus:}
  \uu|_{\bottomBoundary\cup\talusBoundary}=0 , \quad
  \vv|_{\bottomBoundary} &=0,
                           \label{eq:bc.2}
  \\
  % \intertext{and slip condition on talus:}
  \varepsilon\gradx\vv \cdot \mathbf{n_x}|_\talusBoundary &=0,
                                                            \label{eq:bc.3}
\end{align}
where $\gs$ represents the wind stress, $\mathbf{n_x}$ is the
horizontal part of the normal vector.

The limit of the Hydrostatic Equations
(\ref{eq:NS.a})--(\ref{eq:NS.c}) when $\varepsilon\to 0$ is studied on
rigorous mathematical grounds in \cite{Besson-Laydi:92} (for the
stationary case) and \cite{Azerad-Guillen:01} (for the evolutive
case).  Most of existence and regularity results (see
e.g.~\cite{Chacon-Guillen:00,Chacon-RodriguezGomez:05,Cao-Titi:2007,CushmanRoisin-Beckers:09,Guillen-Redondo:ViscSplittingPE:2017})
% excepting~\cite{Azerad:1994,Azerad:PhD:96}
for~(\ref{eq:NS.a})--(\ref{eq:NS.c}), and also the major part of the
associated numerical schemes (see
e.g.~\cite{Chacon-Guillen:00,Chacon-RodriguezGomez:05}) are based on
the introduction of an equivalent \emph{integral-differential
  problem}, by doing a vertical integration of the vertical momentum
equation (\ref{eq:NS.b}). From the numerical point of view, this idea
has advantages (it is only necessary to compute a $2D$ pressure, only
defined in the surface $S$) but also some drawbacks (for instance,
standard FE in unstructured meshes, variable density and
non-hydrostatic cases are difficult to handle).

%%%%%%%%%%%%%%%%%%%%%%%%%%%%%%%%%%%%%%%%%%%%%%%%%%%%%%%%%%%%%%%%%%%%%%%%%%%%%%
%% Redefine \vv and \bv (in this part, they don't denote vertical velocity) %%
%%%%%%%%%%%%%%%%%%%%%%%%%%%%%%%%%%%%%%%%%%%%%%%%%%%%%%%%%%%%%%%%%%%%%%%%%%%%%%
\renewcommand{\vv}{\upsilon}
\renewcommand{\bv}{\overline\upsilon}
%%%%%%%%%%%%%%%%%%%%%%%%%%%%%%%%%%%%%%%%%%%%%%%%%%%%%%%%%%%%%%%%%%%%%%%%%%%%%%

In this work we are concerned on the linear steady model related
to~(\ref{eq:NS.a})--(\ref{eq:NS.c}) and in the less favorable limit
case, $\varepsilon=0$.
Results shall be extended to the non-hydrostatic
case $\varepsilon>0$ in further works.
The case $\varepsilon=0$ is known as Hydrostatic Stokes equations and
its mixed variational formulation reads: find
$(\uu,\vv,\pp)\in\UU\times\V\times\PP$ such that
\begin{alignat}{2}
  \nu (\grad\uu,\grad\buu) - (\pp, \divx\buu) &= (\ff,\buu) + (\gs,\buu)_{\Gamma_s}    &\quad \forall\,\buu\in\UU,
  \label{eq:HS.weak.a}
  \\
  (\pp,\dz\bv) & = 0 &  \quad \forall\,\bv\in\V,
  \label{eq:HS.weak.b}
  \\
  (\div(\uu,\vv),\bp) &= 0 &\quad \forall\,\bp\in\PP.
  \label{eq:HS.weak.c}
\end{alignat}
% \begin{alignat}{2}
%   \nu (\grad\uu,\grad\buu) - (\pp, \divx\buu) &= (\ff,\buu) + (\gs,\buu)_{\Gamma_s}    &\ \forall\,\buu\in\UU,
%   \label{eq:non-HS.weak.a}
%   \\
%   \varepsilon (\grad\vv,\grad\bv) - (\pp,\dz\bv) & = 0 &  \ \forall\,\bv\in\V,
%   \label{eq:non-HS.weak.b}
%   \\
%   (\div(\uu,\vv),\bp) &= 0 &\ \forall\,\bp\in\PP.
%   \label{eq:non-HS.weak.c}
% \end{alignat}
Here $(\cdot,\cdot)$ is the $L^2(\Omega)$ scalar product,
$(\cdot,\cdot)_{\Gamma_s}$ is the $L^2(\Gamma_s)$ scalar product and we define
\begin{align*}
  \UU&=\Uspace=\Big\{\uu\in H^1(\Omega)^{2} {\ | \ }
       \uu|_{\bottomBoundary\cup\talusBoundary}=0\Big\},
       % \label{eq:Uspace}
  \\
  \V&=\Vspace =\Big\{\vv\in L^2(\Omega){\ | \ } \dz\vv\in L^2(\Omega),\
       \vv|_{\surfaceBoundary\cup\bottomBoundary}=0\Big\},
       % \label{eq:Vspace}
  \\
  \PP&=\Pspace=\Big\{\pp\in L^2(\Omega){\ | \ } \int_\Omega\pp=0
       \Big\}.
       % \label{eq:Pspace}
\end{align*}
The space $\UU$ is endowed with the norm $\|\grad\uu\|$ (hereafter $\|\cdot\|$
denotes the $L^2(\Omega)$-norm) while in $\V$ we consider
$\|\dz\vv\|$, which is a norm owing to the homogeneous Dirichlet
condition  $v|_{\Gamma_s\cup\Gamma_b}=0$ and a vertical Poincar\'e
inequality.

As stated
in~\cite{Guillen-RRGalvan:hydrostatic:NUMAT:15,Azerad:1994}, well-posedness
of~(\ref{eq:HS.weak.a})--(\ref{eq:HS.weak.c})
hinges on the following inf-sup conditions:
\begin{align}
  \tag*{\ensuremath{(IS)^\PP}}
  \
  \ConstISp &\|\pp\| \le
  \sup_{0\neq(\uu,\vv)\in \UU \times \V}
  \frac{(\div(\uu,\vv),\pp)}{\|(\grad \uu,\, \dz \vv)\|}
  \qquad
  \forall \,\pp \in \PP,
  \label{eq:ISp}
  \\\noalign{\smallskip}
  \tag*{\ensuremath{(IS)^\V}}
  \
  &\|\dz\vv\| \le
  \sup_{0\neq \pp \in \PP}
  \frac{(\dz \vv,\pp)}{\|\pp\|}
  \qquad
  \forall\, \vv\in \V,
  \label{eq:ISv}
\end{align}
where $\|(\grad \uu,\, \dz \vv)\|$ is the norm of $\UU\times\V$. Note
that~\ref{eq:ISp} is basically the well-known LBB condition
while~\ref{eq:ISv} is a new \textit{hydrostatic} restriction.

% We focus in the less favorable limit case, $\varepsilon=0$, known as
% Hydrostatic Stokes equations:
% \begin{alignat}{2}
%   \nu (\grad\uu,\grad\buu) - (\pp, \divx\buu) &= (\ff,\buu) + (\gs,\buu)_{\Gamma_s}    &\quad \forall\,\buu\in\UU,
%   \label{eq:HS.weak.a}
%   \\
%   (\pp,\dz\bv) & = 0 &  \quad \forall\,\bv\in\V,
%   \label{eq:HS.weak.b}
%   \\
%   (\div(\uu,\vv),\bp) &= 0 &\quad \forall\,\bp\in\PP.
%   \label{eq:HS.weak.c}
% \end{alignat}

% In the discrete case, let ${\cal T}_h$ be a regular family of meshes
% in $\overline\Omega$ satisfying the usual regularity condition:
% there exists $\sigma>1$ such that $h_T\le \sigma \rho_T$ for every
% $T\in{\cal T}_h$, where $h_T$ is the diameter of the triangle $T$
% and $\rho_T$ is the maximum diameter of all circles contained in
% $T$.  Note that no kind of structure is assumed in ${\cal T}_h$ (in
% particular, it is not needed for vertical integration).

In the discrete setting, it was shown in
\cite{Guillen-RRGalvan:hydrostatic:NUMAT:15} (see
also~\cite{Azerad:1994}) that the discrete counterpart of inf-sup
condition~\ref{eq:ISp} is no longer sufficient for
%(uniformly when $\varepsilon\to 0$)
stability of standard conforming FE approximations
of~(\ref{eq:HS.weak.a})--(\ref{eq:HS.weak.c}), because it is also
necessary to choose FE spaces satisfying the discrete counterpart
of~\ref{eq:ISv}. Unfortunately, standard Stokes FE like Taylor-Hood
$\P2$--$\P1$ or ($\P1$+bubble)--$\P1$ do not satisfy~\ref{eq:ISv}.
Thus different FE must be considered (for instance, by approximation
of vertical velocity in a space other than horizontal velocity,
see~\cite{Guillen-RRGalvan:hydrostatic:NUMAT:15,Guillen-RRGalvan:stokes:APNUM:16}).
% case where $\Uh\subset \UU$,
% $\Vh\subset \V$ and $\Ph\subset \PP$ are conforming FE spaces, i

A different idea was introduced
in~\cite{Guillen-RRGalvan:stabilized:SINUM:15}, where
discrete~\ref{eq:ISv} is avoided by adding a consistent stabilizing
term to the vertical momentum equation (\ref{eq:HS.weak.b}). In this
way, the stability for Stokes-LBB FE combinations is shown and error
estimates are provided for Taylor-Hood $\P2$--$\P1$ FE and
mini-element ($\P1$bubble)--$\P1$ approximations, showing optimal
convergence order in the $\P2$--$\P1$ case.

The current work introduces a third approach that, until now, has not been
explored: using Discontinuous Galerkin (DG) methods, we can define
approximations that, without any stabilization, satisfy (in some
sense) both~\ref{eq:ISv} and~\ref{eq:ISp} restrictions.

  DG methods, which are well suited for the construction of stable
  discretizations of compressible (advection-dominated) flows and in
  general for hyperbolic operators, have been extended also for
  incompressible flows~(for a review, see
  e.g.~\cite{CockburnKarniadakisShu:2011:DGM:2408658,arnold_cockburn_unified_2002,di_pietro_ern_2012}
  and references therein) and in general for elliptic operators.  More
  in detail, DG methods for second order elliptic operators can be
  split roughly into two groups: first, the so called Local
  Discontinuous Galerkin (LDG) schemes, where the operator is
  converted into a system of first order equations and numerical
  fluxes are devised as in hyperbolic
  equations~\cite{arnold_cockburn_unified_2002}. On the other hand the
  schemes augmenting the elliptic operator by penalizing the
  discontinuities of the shape
  functions~\cite{glowinski_interior:douglas_dupont:1976,arnold_interior_1982}. These
  latter schemes are known as Interior Penalty (IP) DG method (SIP DG
  methods in the usual symmetric case).

In the same way, DG schemes for compressible (and for incompressible)
flows can be split into two groups: some of them are based in the LDG
schemes~\cite{cockburn_local_2002} while other discretizations are
based on the IP
method~\cite{hansbo_discontinuous_2002,di_pietro_ern_2012}. The scheme
presented here is based in the latter methods and our main
contribution is in the design of a SIP DG scheme where, somehow, the
Hydrostatic restriction~\ref{eq:ISv} is verified, in addition to the
LBB-like restriction~\ref{eq:ISp}.

This paper is structured as follows: in
Section~\ref{sect:ip-dg-formulation} we fix notation an introduce some
useful results from SIP DG approximation of diffusion equations. In
section \ref{sec:dg-discr-hydr} we introduce a SIP DG approximation
for~(\ref{eq:HS.weak.a})--(\ref{eq:HS.weak.c}) where the velocity
field and the pressure unknowns are defined by the same $k$--degree
discontinuous $\P{k}$ polynomials. In
Section~\ref{sect:well-posedness} we show well-posedness for this
approximation of the Hydrostatic Equations and in
Section~\ref{sect:numerical-tests} some numerical tests are shown
which agree with the theory.

\section{SIP DG Approximation of Diffusion Equations}
\label{sect:ip-dg-formulation}
We start fixing notations and
collecting some results which shall be useful in following sections.
Let us denote by $\Th$ a family of meshes of
the domain $\Omega\subset\Rset^d$ ($d=2$ or $3$ in practice) into
non-degenerate disjoint simplicial elements $K$ satisfying usual
regularity assumptions\cite{Ciarlet:78}:
$$
\exists \rho>0 \ /\ \rho\, h_K \le r_K \quad \forall K\in\Th,
$$
where $h_K$ is the diameter of $K$ and $r_K$ is the radius of the
largest ball inscribed in $K$. The number of edges (faces) of the
elements is denoted as $N_\partial$.  Note that more general meshes
can also be handled, specifically elements $K$ are not required to be
simplicial elements and $\Th$ can be any shape and contact-regular
mesh, see e.g.~\cite{di_pietro_ern_2012}, Section 1.4.

We associate to each triangulation $\Th$ the set of interior faces
(edges in $2D$) $\Ehi$ and the set of boundary faces $\Ehb$, defined
as follows: $e\in\Ehi$ if there are two polyhedra $K^+$ an $K^-\in\Th$
such that $e=K^+\cap K^-$ and $e\in\Ehb$ if there is $K\in\Th$ such
that $e=\partial K \cap \partial\Omega$. We define $\Eh=\Ehi\cup\Ehb$.

Let $u$ be a scalar-valued function on $\Omega$ and assume that $u$ is
smooth enough to admit on all $e\in\Ehi$ a (possibly two-valued)
trace. We define the jump and the average of $v$ on $e\in\Ehi$, denoted respectively
as $\jump{u}$ and $\average{u}$,  as follows: if $e=K^+\cap K^-$, then
\begin{equation*}
  \jump{u}_e= u|_{K^+}-u|_{K^-} \quad \hbox{and} \quad
  \average{u}_e = \frac12\Big(u|_{K^+}+u|_{K^-}\Big).
\end{equation*}
If $e\in\Ehb$, we define $\jump{u}=\average{u}=u|_e$.

Let us define the following \textit{broken discrete} Sobolev space, for each $m\ge 0$,
\begin{equation*}
  H^m(\Th) = \big\{ u\in L^2(\Omega) \ |\ u\in H^m(K)\ \forall K\in\Th \big\},
\end{equation*}
the broken gradient operator $\gradh$ for each
$u\in H^1(\Th)$,
$$
(\gradh u)|_K=\grad(u|_K)\quad \forall\, K\in\Th
$$
% Note that $H^m(\Omega)\subset H^m(\Th)$ for all $m\ge 0$
% and if $\gradh\uu=\grad\uu$ if $u$
and the following finite-dimensional subspace of $H^m(\Th)$, composed
of polynomials of degree no more than $k$ in each element:
$$
\Pd h^k:=
\left\{
  u\in L^2(\Omega) \ | \ u\in \Polin{k}(K), \ \forall K\in\Th \right \}.
$$
The following discrete trace
inequality shall be useful: for all $\uh\in \Pd h^k$ and $K\in\Th$,
%  and for every
% edge (face in 3d) $e$ of $K$,
\begin{align}
  % \label{eq:trace.edge}
  % h_K^{1/2} \norm[L^2(e)]{\uh|_e} \le \Ctrace \, \norm[L^2(K)]{\uh},
  % \\
  \label{eq:trace.boundary.element}
  h_K^{1/2} \norm[L^2(\partial K)]{\uh|_K} \le \Ctrace \, \norm[L^2(K)]{\uh},
\end{align}
where $\Ctrace$ is a constant
independent of $h$ and $K$ (and depending on $k$, $d$ and $\rho$).
See e.g.~\cite{di_pietro_ern_2012}, Lemma 1.46 and Remark
1.47, for details.
% see~\cite{di_pietro_ern_2012}, Remark~1.47, also Lemmas 1.41
% and~1.46).
From this inequality, one has the following technical result.
\begin{lemma}
  \label{lemma.technical.inequalities}
  For every $\ph\in \Pd h^k$, the following inequalities are satisfied, for
  constant $C>0$ independent of $h$ (and dependent on $k$, $d$
  and $\rho$):
  \begin{align}
    \label{eq:technical.average.bound}
    &\Big( \sum_{e\in \Eh} h_e \int_e \average{\ph}^2 \Big )^{1/2} \le C \norm[L^2(\Omega)]{\ph},
    \\
    \label{eq:technical.jump.bound}
    &\Big( \sum_{e\in \Eh} h_e \int_e \jump{\ph}^2 \Big )^{1/2} \le C \norm[L^2(\Omega)] \ph.
  \end{align}
\end{lemma}
\begin{proof}
  First we observe that, for all $e\in \Ehi$ with
  $e= \partial K_1 \cap \partial K_2$, one has
  $\jump{\ph}^2 \le 2(p_1^2 +p_2^2)$ and
  $  \average{\ph}^2 \le (p_1^2 + p_2 ^2)/2 $, where
  $p_i = \ph|_{K_i}$, $i\in \{1,2 \}$.
  Also if $e\in \Ehb = \Eh \cap \partial \Omega$, by definition,
  $\jump{\ph}^2=p_1^2 $ and $\average{\ph}^2=p_1^2$.
  In any case, we can write
  \begin{equation*}
    \jump{\ph}^2 \le C_* ( p_1^2 +p_2^2)
    \quad \text{and} \quad \average{\ph}^2 \le  C_* (p_1^2 + p_2^2),
  \end{equation*}
  for some constant $C_*>0$. Therefore
  \begin{align*}
    \sum_{e\in \Eh} h_e \int_e \average{\ph}^2
    &\le C_* \sum_{e\in \Eh} h_e \int_e (p_1^2+p_2^2)
    \\
    &\le
    2 C_* \sum_{K\in \Th} h_K \sum_{e\in \Eh \cap \partial K} \int_e \ph^2
     \le
      2N_\partial C_* \sum_{K\in \Th}  h_K \int_{\partial K} \ph^2
      \le C \norm{\ph}^2,
    \end{align*}
    where $C=2N_\partial C_* C_{tr}$ and $C_{tr}$ is the constant
    introduced in~(\ref{eq:trace.boundary.element}). Inequality
    (\ref{eq:technical.jump.bound}) can be shown similarly.
\end{proof}

Let us consider the following symmetric interior
penalty (SIP) bilinear form  for discontinuous FE
approximation of second order elliptic and parabolic equations:
\begin{equation}
\label{eq:asip}
  \begin{split}
    \asip[sip,\eta](u,\bu) = \int_\Omega \gradh u \cdot\gradh\bu -
    \sum_{e\in\Eh} \int_e\big( \average{\gradh u}\cdot \nn_e \jump{\bu}
    + \jump{u} \average{\gradh\bu}\cdot \nn_e \big) \\
    +
    \eta\sum_{e\in\Eh} \frac1{h_e} \int_e \jump{u}\jump{\bu},
  \end{split}
\end{equation}
for each $u,\bu\in \Pd h^k$. Here $\nn_e=(\nn_\xx,n_z)\in \Rset^d$
denotes the normal vector (in a fixed chosen sense) across the edge or
face $e$, $h_e$ is the diameter of $e$ and $\eta > 0$ is a constant.
The second term at RHS of (\ref{eq:asip}) arises for consistency and
symmetry, while the last one introduces a penalization on interior
faces and boundary faces which enforces coercivity. Also boundary
values are penalized which, assuming Dirichlet boundary conditions, is
used impose weakly these conditions.
Indeed, one has the following
coercivity result in $\Pd h^k$ (see e.g~\cite{di_pietro_ern_2012},
Lemma 4.12) for the norm
\begin{equation*}
  \normsip{u} = \big(\, \norm{\gradh u}^2 + \seminormU{u}^2 \;\big) ^{1/2},
  \quad
  \text{where} \quad \seminormU{u} =
  \Big (\sum_{e\in\Eh} \frac1{h_e} \int_e {\jump u}^2 \Big ) ^{1/2}.
\end{equation*}
\begin{lemma}[Coercivity for $\asip(\cdot,\cdot)$] \label{le:asip.coercivity}
  Let us denote $\eta_*=\Ctrace^2$, with $\Ctrace$ given in \eqref{eq:trace.boundary.element}. For all $\eta>\eta_*$, one has
  \begin{equation}
    \label{eq:asip.coercivity}
    \asip[sip,\eta](\uh,\uh) \ge C(\eta) \normsip{\uh}^2, \quad \forall\,\uh\in \Pd h^k,
  \end{equation}
  where $C(\eta)= (\eta-\eta_*)/(1+\eta)$.
  \label{lemma:sip:coercivity}
\end{lemma}
% \begin{remark}[Sólo para Vicky, Kisko y Rafa]
%   $\normsip{\cdot}$ is really a norm in $\Pd h^k$: if $\normsip{u}=0$ then
%   $$\norm{\gradh u}^2=0\quad\hbox{and}\quad\seminormU{u}^2=0.
%   $$
%    Second equality implies $u\in C^0(\overline\Omega)$ (hence $u\in H^1(\Omega)$) and  also   $u=0$ on $\partial\Omega$. Applying first equality,
%   $\grad u = \gradh u=0$, therefore $u=0$ in $\overline\Omega$.
% \end{remark}
One also has  boundedness on $\Pd h^k$ (see e.g. \cite{di_pietro_ern_2012}, Lemmas~4.16 and~4.20):
\begin{lemma}[Boundedness]
  There is $C_{\text{bnd}}>0$ independent of $h$ (and depending on $\eta$)
  % (in fact, $C_{\text{bnd}}=2+\eta+\eta_*$)
  such that
  \begin{equation*}
    \asip[sip,\eta](\uh,\buh) \le C_{\text{bnd}} \normsip{\uh}\normsip{\buh}, \quad \forall\,\uh,\buh\in \Pd h^k.
  \end{equation*}
  \label{lemma:sip:boundedness}
\end{lemma}

\section{SIP DG Discretization of the Hydrostatic Stokes Equations}
\label{sec:dg-discr-hydr}

In this section we introduce the discrete variational formulation for
the Hydrostatic problem~(\ref{eq:HS.weak.a})--(\ref{eq:HS.weak.c})
based on an SIP DG approximation. For simplicity, homogeneous
Dirichlet boundary conditions are considered (in particular we take
$\gs=0$) although  Neumann conditions can also be imposed in practice,
as outlined in Section~\ref{sect:numerical-tests}.

Here we introduce the same polynomial order for the velocity field
$\wwh=(\uuh,v_h)$ and the pressure $\ph$ spaces:
\begin{align*}
  \UUh &= (\Pd h^k)^{d-1}, \quad \Vh=\Pd h^k, \quad \WWh=\UUh\times\Vh=(\Pd h^k)^d,
  \\
  \Ph &= \Pd h^k.
\end{align*}
% We also define $\Pho$ as the subspace of zero-mean functions in $\Ph$.
%In order to introduce anisotropic viscosity in the SIP formulation,
%let us denote for each $\upsilon,\overline\upsilon\in\Vh$ and
%$\lambda\in\Rset$:
%\begin{equation*}
%  \begin{aligned}
%    \asip[sip,\lambda,\eta](\upsilon,\overline\upsilon)
%    = \lambda & \Big( \int_\Omega \gradh\upsilon\gradh\overline\upsilon
%    - \sum_{e\in\Eh} \int_e \big( \average{\gradh\upsilon}\cdot n_e \jump{\overline\upsilon}
%    + \jump{\upsilon} \average{\gradh\overline\upsilon}\cdot n_e \big)
%    \Big)
%    \\
%    &+ \sum_{e\in\Eh} \frac{\eta}{h_e} \int_e \jump{\upsilon}\jump{\overline\upsilon}.
%  \end{aligned}
%\end{equation*}
%Note that the penalization parameter has been renamed to $\eta$, emphasizing that
%it differs from the constant $\eta$ appearing in~(\ref{eq:asip}).
%%%%%%%%%%%%%%%%%%%%%%%%%%%%%%%%%%%%%%%%%%%%%%%%%%%%%%%%%%%%%%%%%%%%%%%%
%% Redefine \vv and \bv (in this part, they denote vertical velocity) %%
%%%%%%%%%%%%%%%%%%%%%%%%%%%%%%%%%%%%%%%%%%%%%%%%%%%%%%%%%%%%%%%%%%%%%%%%
\renewcommand{\vv}{v}%
\renewcommand{\bv}{\overline v}%
%%%%%%%%%%%%%%%%%%%%%%%%%%%%%%%%%%%%%%%%%%%%%%%%%%%%%%%%%%%%%%%%%%%%%%%%%
For each $\wwh=(\uuh,\vh)$ and $\bwwh=(\buuh,\bvh) \in\WWh$, with
$\uuh=(u_i)_{i=1}^{d-1}$ and $\buuh=(\overline u_i)_{i=1}^{d-1}$, we define the following
bilinear form associated to~(\ref{eq:HS.weak.a})--(\ref{eq:HS.weak.c}):
\begin{equation}
  \label{eq:bilinear.a}
  a_h(\wwh,\bwwh) = \nu\Big( \sum_{i=1}^{d-1} \asip[sip , \eta](u_i,\overline u_i)
  + \eta\sum_{e\in\Eh} \frac1{h_e} \int_e \jump{\vv_h\nz}\jump{\bv_h\nz} \Big),
\end{equation}
where SIP bilinear form for vertical velocity is not introduced  (due
to the lack of diffusive terms in vertical momentum equations)
although a penalization term for $\vh$, in vertical direction, is present
in interior and boundary faces.
%where
%\begin{equation*}
%  a_h^{u}(\uh,\buh) = \sum_{i=1}^{d-1} \asip[sip,\nu , \eta](u_i,\overline u_i), \quad
%  a_h^{v}(\vh,\bvh) = \asip[sip,\varepsilon, \eta](\vh,\bvh).
%\end{equation*}

The next step consists in introducing a suitable norm on $\wwh=(\uuh,\vh) \in\WWh$ for
which a generalized coercivity result can be obtained. Note that
using Lemma~\ref{lemma:sip:coercivity} we have only
\begin{equation}
a_h(\wwh,\wwh) \ge \nu \big(C_\eta \normsip{\uuh}^2 + \eta\seminormV{\vh}^2  \big),
\label{eq:ah.partial.coercivity}
\end{equation}
with
$$\normsip{\uuh}^2=\sum_{i=1}^{d-1}\normsip{u_i}^2, \quad
\seminormV{\vh}^2 = \sum_{e\in\Eh} \frac1{h_e} \int_e \jump{\vh \nz}^2.
$$
If we define
the following ``isotropic'' velocity norm
$$
\norm[\mathrm{iso}]{\wwh} =
\left(\normsip{\uuh}^2+\normsip{\vh}^2\right)^{1/2},
$$
then no
% coercivity of $a_h(\cdot,\cdot)$ is not satisfied. Specifically,
%noticing that for $\eta=\nu\eta$ one has
%$\asip[sip,\nu](u_i,\overline u_i) = \nu\asip(u_i,\overline u_i)$ and
% using Lemma~\ref{lemma:sip:coercivity} we have only
% $$
% a_h(\wwh,\wwh) = \nu\sum_{i=1}^{d-1} \asip[sip , \eta](u_i,\overline u_i) \ge \nu\, C_\eta \normsip{\uh}^2,
% $$
% but no
control for $\normsip{\vh}^2$ can be obtained.
% (specifically, for horizontal derivatives of $\vh$).
% Hence exists $C_{\eta,\nu,\varepsilon}>0$ such that
% $a_h(\wwh,\bwwh) \ge C_{\eta,\nu,\varepsilon} \norm[\mathrm{vel,iso}]{\wwh}^2$ but we
% cannot ensure that $C_{\eta,\nu,\varepsilon}$ does not vainsh when
% $\varepsilon\to 0$.
In order to avoid this obstacle, we introduce the
following anisotropic or hydrostatic velocity norm:
\begin{equation*}
  \normvel{\wwh} = \left(\normsip{\uuh}^2 + \norm{\dzh\vh}^2 + \seminormV{\vh}^2\right)^{1/2},
\end{equation*}
where $\dzh$ is the broken vertical derivative (which is defined
similarly to $\gradh$). % As a first step to show generalized coercivity
% of $a_h(\cdot,\cdot)$, we shall prove an inf-sup condition.
Although inequality~(\ref{eq:ah.partial.coercivity}) does not allow to
infer the  coercivity of $a_h(\cdot,\cdot)$ for $\normvel{\cdot}$ (due to the lack of
control for $\dzh\vh$), one has the following inf-sup bound for
$\norm{\dzh\vh}$ in terms of $P_h$:
\begin{lemma} [Stability for $\dzh\vh$]
It holds
  \begin{equation}
    \label{eq:vh.stability}
     \norm{\dzh\vh} =
    \sup_{\bph \in \Ph}\frac{\int_{\Omega} \bph \, \partial_{z,h}\,\vh}{\|\bph\|}
     \quad \forall\vh\in\Vh.
  \end{equation}
  \label{lemma:vh.stability}
\end{lemma}
\begin{proof}
  Given $\vh\in\Vh$, it suffices to note that the supremum of \eqref{eq:vh.stability} is
  reached for
  $\bph= \partial_{z,h} \vh \in \Ph \subset
  L^2(\Omega) $.
\end{proof}

\begin{remark}
  In general, $\partial_{z,h} \vh \not \in L_0^2(\Omega)$, thus previous
  result is not clear taking supreme on zero-mean discrete pressures $P_h\cap L_0^2(\Omega)$.
\end{remark}

At this point, well-posedness of the discrete problem hinges on a
bound of $\norm{\ph}$. The problem is that, for general (non
zero-mean) pressures, it cannot be obtained by the well-known discrete
inf-sup (or LBB) condition. For this reason, a specific DG Galerkin
inf-sup condition for bounding $\norm{p-\mean{p}}$ is now introduced
(where $\mean{p}$ denotes the mean of $p$ in $\Omega$).  Let us define the
following discrete bilinear form:
$$ b_h(\wwh, \ph )= -\int_{\Omega} \ph\, \gradh \cdot \wwh + \sum_{e\in\Eh} \int_e \jump{\wwh} \cdot \nn_e\, \average{\ph}$$
where $ \nabla_h \cdot$ is the ``broken'' divergence operator
(defined on each $K\in\Th$). It is not difficult
to show continuity for $b_h(\wwh,\ph)$ with $\normvel{\wwh}$ and $\norm{\ph}$ in $\WWh \times \Ph$.
The following property is also satisfied:
\begin{equation}
  \label{eq:b(.,1)=0}
  b_h(\wwh,1) = 0 \qquad \forall\,\wwh\in\WWh.
\end{equation}
Let us consider the following pressure seminorm in
$ H^1(\Th ) \supset \Pd h^k$:
$$
|p|_P  = \Big ( \sum_{e\in\Ehi} h_e \, \| \jump{p} \|_{L^2(e)}^2 \Big )^{1/2}.
$$

\begin{lemma} [Stability for $\ph$] There exists $\ConstISph > 0$ independent of $h$, such that
  \begin{equation}
    \ConstISph  \, \norm{\ph-\mean{\ph}}
    \le \sup_{\wwh \in \WWh \setminus \{0\}} \frac{b_h ( \wwh, \ph )}{\normvel{\wwh}}
    + | \ph |_P , \qquad \forall \ph \in \Ph.
    \label{eq:ph.stability}
  \end{equation}
  \label{lemma:ph.stability}
\end{lemma}

\begin{proof}
  Let $\ph\in\Ph$. It is known (see e.g. \cite{di_pietro_ern_2012},
  Lemma 6.10) that inequality (\ref{eq:ph.stability}) holds for
  zero-mean pressures if $\normvel{\wwh}$ is replaced by
  $\norm[\mathrm{iso}]{\wwh}$. On the other hand, for all
  $ \vh \in \Vh$, we have
  $ \| \partial_{z,h} \vh \| ^2 \le \| \grad _h \vh \|  ^2$
  and then $\normvel{\wwh} \le \norm[\mathrm{iso}]{\wwh}$ for all
  $\wwh \in \WWh$.  Therefore, exists $\ConstISph>0$ such that
  \begin{align*}
    \ConstISph \, \norm{\ph-\mean{\ph}}
    &\le \sup_{\wwh \in \WWh
      \setminus \{0\}} \frac{b_h(\wwh,
      \ph-\mean{\ph})}{\norm[\mathrm{iso}]{\wwh}} + | \ph-\mean{\ph}
    |_P
    \\
    &\le \sup_{\wwh \in \WWh \setminus \{0\}} \frac{b_h ( \wwh, \ph
    )}{\normvel{\wwh}}+ | \ph |_P.
\end{align*}
\end{proof}

% Finally, we are going to prove  a result related to a partial coercivity of $ a_h (\cdot, \cdot) $. Using the previous
%  result, we can prove the following Lemma:

% \begin{lemma}
% There exist $\overline{\eta} $ and $ \alpha >0$  (independent of $h$ and  $\varepsilon$), such that
% $ \forall \eta > \overline{\eta}$,

% \begin{equation}\label{partial-coercivity-ah}
%  a_h ( \ww, \ww ) \ge \alpha \Big ( \normvel{\uh}^2
% % + \sum_{i=1}^{d -1} ( \| \grad _h u _i \| _{L^2} ^2
%   + \seminormU{u_i}^2 )
% + \seminormU{v}^2  \Big )
% \end{equation}

% \end{lemma}

% {\bf Proof:}
% For all $ {\bf w}= ({\bf u},v) \in \wwh$,
% $$ a_h ( {\bf w}, {\bf w} ) = \nu \sum_{i=1}^{d -1} \asip[sip , \eta]( u_i, u_i ) + \varepsilon\,  \asip[sip , \eta/\varepsilon](v,v)
% \ge  \sum_{i=1}^{d -1} a_h^{\nu,\eta} ( u_i, u_i ) +  \sum_{e\in\Eh} \frac{\eta}{h_e}
% \int _e \jump{v} \, \jump{v}
% $$
% Now using  Lemma 1,
% $$
% a_h ( {\bf w}, {\bf w} )
% \ge \nu \, C_\eta \, \Big  ( \sum_{i=1}^{d -1} ( \| \grad _h u _i \| _{L^2} ^2 + \seminormU{u_i}^2) ) \Big ) + \seminormU{u_i}^2) ) \Big ) +
%  \overline{\eta} \, \seminormU{v}^2
%  $$
% hence we can conclude the proof.

% \

To conclude this section, let us formulate the following DG
discretization of the Hydrostatic-Stokes
problem~(\ref{eq:HS.weak.a})--(\ref{eq:HS.weak.c}): find
$(\wwh, \ph) \in \WWh \times \Ph$ such that
\begin{equation}\label{disc-var-pb}
\left \{
\begin{array}{l}
  \displaystyle
  a_h ( \wwh, \bwwh) +b_h ( \bwwh, \ph ) = \int_{\Omega} \ff \cdot \buuh
  % + \int_{\Gamma_s} \gs \cdot \buuh
  , \quad \forall \,\bwwh=(\buuh,\vh) \in \WWh,
  \\\noalign{\medskip}
  \displaystyle
  -b_h (\wwh, \bph) + s_h (\ph, \bph) + \pEpsilon \int_\Omega \ph\bph=0 \quad \forall\, \bph \in \Ph,
\end{array} \right.
\end{equation}
where the stabilization bilinear form
$$s_h (\ph, \bph)=  \sum_{e\in\Ehi} h_e \, \int_e
\jump{\ph} \, \jump{\bph}$$
is introduced to control the $L^2$-norm of $\ph-\mean{\ph}$ (by
Lemma~\ref{lemma:ph.stability}) and $\pEpsilon>0$ is a small
penalization parameter. Note that the choice $\pEpsilon=0$ leads to a
ill-posed system, due to the fact that
\begin{equation}
  \label{eq:s(.,1)=0}
  s_h(\ph,1)=0 \quad \forall \,\ph\in\Ph
\end{equation}
which, together with~(\ref{eq:b(.,1)=0}), means that if $\pEpsilon=0$ and
$(\wwh,\ph)\in\WWh\times\Ph$ is a solution to~(\ref{disc-var-pb}), then
$(\wwh,\ph+C)$ is also in $\WWh\times\Ph$ and it
solves~(\ref{disc-var-pb}).

\section{Well-Posedness of the Discrete Problem}
\label{sect:well-posedness}

The discrete formulation~(\ref{disc-var-pb}) can be rewritten in a
vectorial form as follows: find $(\wwh, \ph) \in \WWh \times \Ph$ such
that
\begin{equation}
  c_h \Big ( (\wwh,\ph), ( \bwwh, \bph) \Big )=\displaystyle\int _{\Omega} \ff \cdot \buuh
  % + \int_{\Gamma_s} \gs\cdot\buuh
  , \qquad
  \forall (\bwwh, \bph) \in     \WWh \times \Ph,
  \label{eq:hydrostatic.reformulation}
\end{equation}
where
\begin{multline}
c_h \Big ( (\wwh,\ph), ( \bwwh, \bph) \Big )= a_h ( \wwh, \bwwh)
+b_h ( \bwwh, \ph )
\\
-b_h (\wwh, \bph) + s_h (\ph, \bph) +\pEpsilon\int_\Omega \ph \bph.
\label{eq:bilinear.c}
\end{multline}
We consider  the following norm in $\Xh =\WWh \times \Ph$:
\begin{multline*}
  \| (\wwh, \ph)\|_{\Xh} = \Big ( \normvel{{\bf w}_h}^2 + \| \ph
  \|^2 + |\ph|_P^2\Big ) ^{1/2}
  \\
  =
  \Big ( \normsip{\uuh}^2 + \| \partial_{z,h} v \|
  ^2 + \seminormV{v} ^2+ \| \ph \|^2 + |\ph|_P^2\Big )^{1/2}.
\end{multline*}
According to Banach-Necas-Babu\v{s}ka theorem (see
e.g.~\cite{Ern-Guermond:04}) well-posedness of discrete
problem~(\ref{disc-var-pb}) hinges on the following discrete stability
result for $c_h(\cdot,\cdot)$.
\begin{theorem}[Discrete inf-sup stability]
  \label{teor.well-posedness}
  Assume that the penalty parameter $\eta$ in $a_h^{sip,\eta}(\cdot,\cdot)$ is such
  that $ \eta > \eta_*$, with $\eta_*$ defined in
  Lemma~\ref{lemma:sip:coercivity}. Then, there is $\gamma >0$
  independent of $h$ and $\pEpsilon$ such that, for all
  $(\wwh, \ph) \in \Xh= \WWh \times \Ph$, one has
  \begin{multline}
    \label{eq:global.inf-sup}
    \sqrt{\pEpsilon}\,\norm{\ph-\mean\ph}
    + \gamma \,\| (\wwh, \ph-\mean{\ph})\|_{\Xh}
    \\
    \le \sup_{(\bwwh,\bph)\in \Xh \setminus
      \{0\}} \frac{c_h ( (\wwh, \ph), (\bwwh,\bph)  ) }{\norm[\Xh]{(\bwwh,\bph)} }
    +
    \pEpsilon|\Omega|{\mean{\ph}^2}.
  \end{multline}
\end{theorem}
\begin{corollary}
If  $(\wwh, \ph) \in \Xh= \WWh \times \Ph$ is a solution of scheme \eqref{disc-var-pb}, in particular $\mean{\ph}=0$ and then \eqref{eq:global.inf-sup} implies
$$
\gamma \,\| (\wwh, \ph)\|_{\Xh}
    \le \sup_{(\bwwh,\bph)\in \Xh \setminus
      \{0\}} \frac{c_h ( (\wwh, \ph), (\bwwh,\bph)  ) }{\norm[\Xh]{(\bwwh,\bph)} }.
$$
Consequently, scheme \eqref{disc-var-pb} is well-posed.
\end{corollary}

\begin{proof}[Proof of Theorem \ref{teor.well-posedness}]
  Let $(\wwh, \ph) \in \Xh$,  let $S(\wwh, \ph)$
  denote the supreme on the right hand side of
  (\ref{eq:global.inf-sup}) and let us introduce the following notation:
  $\Phi \lesssim \Psi$ if $\Phi \le C \, \Psi$ for some constant
  $C>0$ independent of $h$.  Owing to
  (\ref{eq:ah.partial.coercivity}) and also to (\ref{eq:b(.,1)=0}) and~(\ref{eq:s(.,1)=0}),
  \begin{multline*}
    c_h ( (\wwh, \ph), (\wwh,\ph-\mean\ph)  )= a_h ( \wwh, \wwh)+s_h (\ph,\ph)
    +\pEpsilon \int_\Omega\ph (\ph-\mean\ph)
    % \\
    % \gtrsim
    %  \normsip{\uuh}^2
    % + \seminormV{\vh}^2
    % + \pEpsilon(\ph,\ph)
    % - \pEpsilon |\Omega|\mean{\ph}^2
    \\
    \gtrsim
    \normsip{\uuh}^2
    + \seminormV{\vh}^2
    + |\ph|_P^2
    + \pEpsilon\norm{\ph-\mean\ph}^2,
  \end{multline*}
  where we applied the following property: $\int_\Omega \ph (\ph-\mean\ph)=\norm{\ph-\mean\ph}^2$.
  Therefore
  \begin{equation}
    \label{eq:partial-coercivity-ch}
     \normsip{\uuh}^2 + \seminormV{\vh}^2 + |\ph|_P^2
     + \pEpsilon\norm{\ph-\mean\ph}^2
     \lesssim S(\wwh, \ph) \norm[\Xh]{(\wwh,\ph-\mean\ph)}.
  \end{equation}

  The rest of the proof is divided into 3 steps:

  \begin{enumerate}
  \item[\em 1)] to estimate $\| \ph-\mean\ph \|$ uniformly on
    $\pEpsilon$,
  \item[\em 2)] to estimate $\| \partial_{z,h} v\| $, and
  \item[\em 3)] to collect estimates and apply Young's inequality.
  \end{enumerate}

  \noindent{\em Step 1: Estimate of $\|\ph-\mean\ph\| $.}
  %-------------------------------------------------------------------
  %
  It can be obtained arguing as in the Stokes framework. Specifically,
  definition of $c_h(\cdot,\cdot)$ means that, for all
  $ \bwwh\in \WWh$,
  $$
  b_h (\bwwh, \ph)= c_h((\wwh,\ph), (\bwwh,0))-a_h(\wwh,\bwwh),
  $$
  then inf-sup condition~(\ref{eq:ph.stability}) imply
  \begin{align*}
    \ConstISph \, \| \ph-\mean\ph \|
    &\le \displaystyle \sup_{\bwwh
      \in \WWh}\Big (\frac{-a_h(\wwh, \bwwh)}{\normvel{ \bwwh} }+
    \frac{c_h((\wwh,\ph), (\bwwh,0))}{\| (\bwwh,0)\|_{\Xh}} \Big ) + |
    \ph |_P,
    \\
    &\le \displaystyle \sup_{\bwwh \in \wwh}\frac{-a_h(\wwh,
      \bwwh)}{\normvel{\bwwh}}+S(\wwh, \ph)+ |\ph|_P.
  \end{align*}
  Boundedness of $a_h(\cdot,\cdot)$ for $\normvel{\cdot}$, follows
  from Lemma~\ref{lemma:sip:boundedness}, namely
  \begin{align*}
  a_h(\wwh,\bwwh) &\lesssim \sum_{i=1}^{d-1}
  \normsip{u_i}\normsip{\overline u_i} + \eta\sum_{e\in\Eh} \frac1{h_e}
  \int_e \jump{\vv_h\nz}\jump{\bv_h\nz}
  \\
    &\lesssim
      \normsip{\uu_h}\normsip{\buu_h}
      + |v_h|_V |\bv_h|_V,
  \end{align*}
  so that
  \begin{equation}\label{estimation-ph}
    \|\ph-\mean\ph\| \lesssim  \normsip{\uu_h}
+ |v_h|_V+S(\wwh, \ph) + |\ph |_P.
  \end{equation}
  Note that former bound depends on $\normsip{\uu_h} + |v_h|_V$ and not
  on $\norm{\dzh v}$, what now allows bounding $\dzh v$ in
  terms of pressure.

  \bigskip
  \noindent{\em Step 2: Estimate of $ \| \partial_{z,h} v\| $:}
  %-------------------------------------------------------------------
  %
  Definition of $c_h(\cdot,\cdot)$ yields, for all $\bph \in\Ph$,
  $$
  -b_h (\wwh, \bph) =  c_h ( (\wwh, \ph), (0,\bph)  )-s_h(\ph,\bph)-\pEpsilon\int_\Omega \ph \bph.
  $$
  Therefore, from the definition of $b_h(\cdot,\cdot)$:
  \begin{multline*}
    \int_{\Omega} \bph \, \dzh \vh= - \int_{\Omega} \bph\,\nabla_{{\bf x},h} \cdot \uuh
    + \sum _{e\in \Eh} \int_e \jump{\wwh}\cdot n_e \, \average{\bph} +
    \\
    c_h((\wwh,\ph), (0, \bph)) - s_h(\ph,\bph) -\pEpsilon\int_\Omega \ph \bph.
  \end{multline*}
  And inf-sup condition~(\ref{eq:vh.stability}) imply
  $$ \| \dzh \vh\| \le S_1+S_2+S(\wwh,\ph)+S_3 + S_4$$
  where we define
  \begin{gather*}
    S_1= \sup_{\bph \in \Ph} \frac{\int_{\Omega} \bph \, \nabla_{{\bf
          x},h}\cdot \uuh}{\| \bph \|},
    \qquad
    S_2 =\sup_{\bph \in
      \Ph} \frac{\sum _{e\in \Eh}\int_e \jump{\wwh}\cdot \nn_e \,
      \average{\bph}}{\| \bph \|},
    \\\noalign{\medskip}
    S_3=\sup_{\bph \in \Ph} \frac{s_h(\ph,\bph)}{\| \bph \|},
    \qquad
    S_4=\sup_{\bph \in \Ph} \frac{\pEpsilon\int_\Omega \ph\bph}{\| \bph \|} =
    \pEpsilon \| \ph \|.
  \end{gather*}
  % \qquad S_4= \sup_{\bp \in \Ph}\frac{c_h((\wwh,\ph), (0, \bph))}{\| \bph
  %   \|}.

  %% S1 %%
  For $S_1$, it is easy to see that
    $$
    S_1= \| \nabla_{{\bf x},h} \cdot \uuh \|
    \lesssim \| \nabla_{{\bf x},h} \uuh \|.
    $$

  %% S2 %%
  For $S_2$,
  % we observe that, for all $e\in \Ehi$ with
  %   $e= \partial K_1 \cap \partial K_2$, one has
  %   $\frac{1}{2} |\jump{\wwh}|^2 \le |\ww_1|^2 +|\ww_2| ^2$ and
  %   $  \average{\bph}^2 \le (\bp _1^2 + \bp_2 ^2)/2 $, where
  %   $\ww_i = \wwh|_{K_i} , \ \bp_i= \bph|_{K_i}$, $i\in \{1,2 \}$.
  %   Also if $e\in \Ehb = \Eh \cap \partial \Omega$, by definition,
  %   $|\jump{\wwh}|^2=|w_1|^2 $ and $\average{\bph}^2=\bph^2$.
    % \begin{equation}
    %   \label{eq:bound.mean.and.jump}
    %   |\jump{\wwh}|^2 \lesssim ( |\ww_1|^2 +|\ww_2|^2)
    %   \quad \text{and} \quad \average{\bph}^2 \lesssim (\bp_1^2 + \bp_2^2).
    % \end{equation}
    applying Cauchy-Schwarz inequality:
    $$ \sum_{e\in \Eh} \int_e \jump{\wwh} \cdot n_e \, \average{\bph} \le \Big
    ( \sum_{e\in \Eh} \frac{1}{h_e} \int_e |\jump{\wwh}\cdot n_e|^2 \Big ) ^{1/2}\,
    \Big (\sum_{e\in \Eh} h_e \int_e \average{\bph}^2 \Big )^{1/2}:=
    I_1\cdot I_2.
    $$
    One has:
    \begin{align*}
      I_1 &\le \left( \sum_{e\in \Eh} \frac{1}{h_e} \int_e \Big(
      \sum_{i=1}^{d-1} \jump{u_i}^2 +(\jump{\vh}\nz)^2\Big) \right) ^{1/2}
      \\
      &\lesssim
      \left( \sum_{i=1}^{d-1}\sum_{e\in \Eh} \frac{1}{h_e} \|
      \jump{u_i}\|_{L^2(e)}^2 + \sum_{e\in \Eh}\frac{1}{h_e} \|
      \jump{\vh}\nz\|_{L^2(e)}^2 \right) ^{1/2}
       = \left(\sum_{i=1}^{d-1} \seminormU{u_i}^2 + \seminormV{\vh}^2 \right)^{1/2}.
    \end{align*}
    On the other hand, using~(\ref{eq:technical.average.bound}) one has
    $
    I_2 \lesssim \norm{\bph}.
    $
   % For $I_2$, applying second inequality in~(\ref{eq:bound.mean.and.jump})
   % and then local trace inequality~(\ref{eq:trace.boundary.element}),
   %  \begin{align*}
   %    I_2^2  \lesssim
   %    \sum_{e\in \Eh} h_e \int_e (\bph_1^2+\bph_2^2)
   %    \le
   %    2 \sum_{K\in \Th} h_K \sum_{e\in \Eh \cap \partial K} \int_e \bph^2
   %    \le
   %    2N_\partial \sum_{K\in \Th}  h_K \int_{\partial K} \bph^2
   %    \lesssim \norm{\bph}^2,
   %  \end{align*}
   %  {\bf K: explicar un poco mas esta desigualdad}
   %  where $N_\partial$ is the number of edges (faces) in elements.
    Therefore
    $$S_2 \lesssim \Big ( \sum_{i=1}^{d-1} \seminormU{u_i}^2 + \seminormV{\vh}^2
    \Big ) ^{1/2}. %  + \| \ph\|_{L^2 (\Omega)} \}.
    $$
  %% S3 %%
    To bound $S_3$, we apply~(\ref{eq:s(.,1)=0}), then Cauchy-Schwarz and
    inequality~(\ref{eq:technical.jump.bound}):
  \begin{multline*}
    s_h (\ph, \bph) = \displaystyle\sum_{e \in
      \Ehi} h_e \int_e \jump{\ph-\mean\ph} \jump{\bph}
    \\
    \le \Big ( \sum_{e \in \Ehi}
    h_e \int_e \jump{\ph-\mean\ph}^2 \Big )^{1/2}\,\Big ( \sum_{e \in\Ehi} h_e
    \int_e \jump{\bph}^2 \Big)^{1/2}
    % \\
    %  \le \Big ( \sum_{K \in
    %   \Th} h_K \| \ph-\mean\ph\|_{L^2 (\partial K)}^2 \Big )^{1/2}\,
    % \Big ( \sum_{K
    %   \in\Th} h_K \|\bph\|_{L^2 (\partial K)}^2 \Big)^{1/2}
    \\
    \lesssim
    \|\ph-\mean\ph\| \,\|\bph\|.
  \end{multline*}
  Therefore
    $$ S_3 \lesssim \|\ph-\mean\ph\|.$$
    Finally,
    $$
    S_4= \pEpsilon \norm{\ph}
    \le \pEpsilon\norm{\ph-\mean\ph} + \pEpsilon\norm{\mean\ph}
    = \pEpsilon\norm{\ph-\mean\ph} + \pEpsilon|\Omega|\mean\ph^2.
    $$
   Summarizing:
   \begin{multline*}
     \norm{\dzh\vh}\lesssim \| \nabla_{{\bf x},h} \uuh \| +
     \Big ( \sum_{i=1}^{d-1} \seminormU{u_i}^2 + \seminormV{\vh}^2
     \Big ) ^{1/2}
     + S(\wwh, \ph)
     \\
     + \norm{\ph-\mean\ph}
     + \pEpsilon|\Omega|\mean\ph^2,
   \end{multline*}
  % %% S4 %%
  %  Finaly, we have to bound $S_4$. Considering
  %   $\bwwh=0$, we have
  %   \begin{multline*}
  %     \| (0,\bph)\|_{\Xh}^2= \| \bph\|_{L^2}^2+ |\bph|_P^2 = \| \bph
  %     \|_{L^2}^2+ \sum_{e \in\Ehi} h_e \int_e \jump{\bph}^2
  %     \\
  %     \lesssim \| \bph \|_{L^2}^2 + \sum_{K \in\Th} h_K
  %     \int_{L^2(\partial K)}|\bph|^2 \lesssim \|\bph\|_{L^2}^2.
  %   \end{multline*}
  %   Hence
  %   $$ S_4\lesssim  \sup _{\bph \in \Ph }  \frac{c_h \Big (
  %     (\wwh,\ph),(0,\bph) \Big )}{\| (0,\bph)\|_{\Xh}}\le S(\wwh, \ph).$$
that is
\begin{equation} \label{S_4}
 \norm{\dzh\vh}\lesssim
 \normsip{\uu_h} + \seminormV{\vh}    + S(\wwh, \ph)
     + \norm{\ph-\mean\ph}
     + \pEpsilon|\Omega|\mean\ph^2
\end{equation}

  \noindent{\em Step 3:}
  %-------------------------------------------------------------------

  Taking into account~(\ref{eq:partial-coercivity-ch})
  and the above estimates \eqref{estimation-ph} and \eqref{S_4}, one has
  $$
    % \pEpsilon\,\norm{\ph-\mean\ph}^2 +
  \| (\wwh, \ph-\mean\ph) \|_{\Xh}^2 \lesssim A^{\uu, v, p} + B^p + C^{\dz v},$$
  where $A^{\uu, v, p} , \ B^p $ and $ C^{\dz v},$ are defined and bounded as
  follows:
  \begin{align*}
    A^{\uu, v, p} &= \normsip{\uuh}^2 +
                    \seminormV{\vh}^2 +|\ph|_p^2
                    + \pEpsilon\,\norm{\ph-\mean\ph}^2 \lesssim
                    \, S(\wwh, \ph)\, \|(\wwh, \ph-\mean\ph)\|_{\Xh}.
    \\ \noalign{\medskip}
    B^p & =\| \ph - \mean\ph\|_{L^2}^2
          \lesssim \Big (\normsip{\uu_h}+ |v_h|_V +S(\wwh,\ph)+ |\ph |_P  \Big )^2
    \\
    & \qquad \lesssim
          S(\wwh, \ph) \, \|(\wwh, \ph-\mean\ph)\|_{\Xh} +S(\wwh, \ph)^2.
    \\\noalign{\medskip}
    C^{\dz v} &= \| \partial_{z,h} v \|_{L^2} ^2 \lesssim
                \big(
                \normsip{\uuh} + \seminormV{\vh}
                + \| \ph-\mean\ph \|_{L^2}
                + S(\wwh, \ph)
                + \pEpsilon|\Omega|\mean\ph^2
                \big)^2
    \\
    & \qquad \lesssim S(\wwh, \ph) \| (\wwh, \ph-\mean\ph)\|_{\Xh} +S(\wwh, \ph)^2
    + \big(\pEpsilon|\Omega|\mean\ph^2 \big)^2.
  \end{align*}
  % In the above estimate, we have used
  % (\ref{teor.well-posedness}).
  Therefore,
  $$
  \| (\wwh, \ph-\mean\ph)\|_{\Xh}^2
  \lesssim S(\wwh, \ph)\,\| (\wwh, \ph-\mean\ph)\|_{\Xh} +S(\wwh, \ph)^2
  + \big(\pEpsilon|\Omega| \mean\ph^2 \big)^2
  $$
  and the conclusion follows from Young's inequality.
\end{proof}

\section{Numerical Tests}       %
\label{sect:numerical-tests}
We have developed some qualitative numerical tests which are agree
with previous theoretical results. Specifically, we were able to
program a standard lid driven cavity test for the discrete
formulation~\eqref{disc-var-pb} using FreeFem++~\cite{HechtFreeFem++},
a high level PDE language and solver which makes simple to develop
variational formulations. In the first test, we used discontinuous
$\P1/\P1$ for velocity and pressure and introduced the following
parameters: $\Omega = (0,1)^2\subset\Rset^2$, unstructured mesh with
$h\approx 1/30$, horizontal viscosity $\nu=1$, SIP penalization
$\eta=10^2$, pressure penalization $\delta_p=10^{-12}$.

Te following Dirichlet boundary are defined. On surface,
$\surfaceBoundary=\{(x,1), x\in(0,1)\}$: $u(x,y)=x(1-x)$, $v=0$. On on
bottom, $\bottomBoundary=\{(x,0), x\in(0,1)\}$: $u=0$, $v=0$. And on
sidewalls,
$\talusBoundary=\{(x,y)\in\Rset^2, x\in \{0,1\}, y\in(0,1)\}$:
$u=0$. We introduce homogeneous Neumann boundary condition for $v$
on $\talusBoundary$: $\dz v=0$.

Former boundary conditions are fixed weakly.  Specifically, the SIP
bilinear form~(\ref{eq:asip}), utilized in~(\ref{eq:bilinear.a}) for
horizontal components of velocity, and also the jump bilinear term
introduced for $\vv_h$ in~(\ref{eq:bilinear.a}), are modified as
follows: for each term regarding to a Dirichlet boundary edge, a
corresponding term is introduced in the right hand side linear
form. As all boundary conditions are zero except $u|_\surfaceBoundary$,
the only additional terms correspond to:
$$
\sum_{e\in\Ehb\cap\surfaceBoundary} \int_{e}
x(1-x)\, \big(
\grad\buh\cdot\mathbf{n}
+ \frac\eta{h_e}\buh
 - \mathbf{n}_x\bph
\big)ds.
$$
This expression can be simplified even more, considering that
$\textbf{n}_x=0$ on $\surfaceBoundary$.
    % - int1d(Th, surface) ( uSurface*dn(ub) + LocalPenalty*uSurface*ub
    %     		   - uSurface*N.x*pb			   % )
On the other hand, terms related to Neumann boundary edges are
eliminated in~(\ref{eq:bilinear.a}) and a corresponding term is
introduced in the RHS as usual in Neumann boundary conditions. In our
case, we have only a null Neumann condition for $v|_\talusBoundary$.

Resulting velocity field and pressure
iso-values (figure~\ref{fig:P1P1cavity}) reproduce the expected
behavior: velocity recirculation and hydrostatic (vertical) pressure
iso-values. These results are improved for higher polynomial order
approximation, specifically for discontinuous $\P2/\P2$ velocity
approximation (figure~\ref{fig:P2P2cavity}). In this case, a higher
SIP penalization parameter, $\eta=10^4$, must be introduced.

% \subsection{Cavity Test}
% For the discrete problem~(\ref{disc-var-pb}), we
% programmed a classical cavity test, using the following data:
% Physical domain with vertical/horizontal ratio: $\eps=10^{-7}$.
% Viscosity: $\nu=1$. RHS functions: $\ff=0$, $g=0$.
% Adimensional domain $\Omega=[0,1]^2$, structured
% $32\times 32$ mesh ($h \simeq 10^{-2}$).
% $P_1$ finite elements for velocity \& pressure.

% Dirichlet boundary conditions (B.C.): Let
% $\Gamma_s=\mbox{surface boundary}$, we take:
% $$
% u=x(x-1) \text{ on } \Gamma_S, \quad u=0 \text{ on } \partial\Omega \setminus \Gamma_S,
% \quad v=0 \text{ on } \partial\Omega.
% $$
% % \item pEpsilon  = 1.0e-10; // Penalty epsilon in p-equations
% % Boundary conditions are {imposed weakly}, using the Nitsche method.
% % To determinate adequate IP and Nitsche penalty parameters was not easy.
% % After a heuristic  approach We set: IP Penalty parameter $\mu=10^{2}$, B.C. Penalty $\eta = 10^{2}$.

\begin{figure}
  \begin{center}
    \begin{tabular}{@{}cc@{}}
      \includegraphics[width=0.35\linewidth, height=0.37\linewidth]{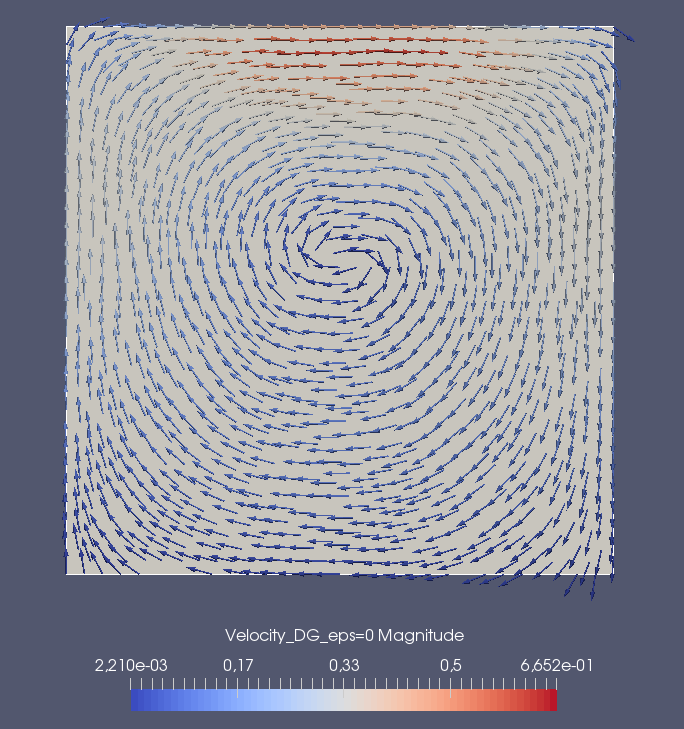}
      &
        \includegraphics[width=0.35\linewidth,height=0.37\linewidth]{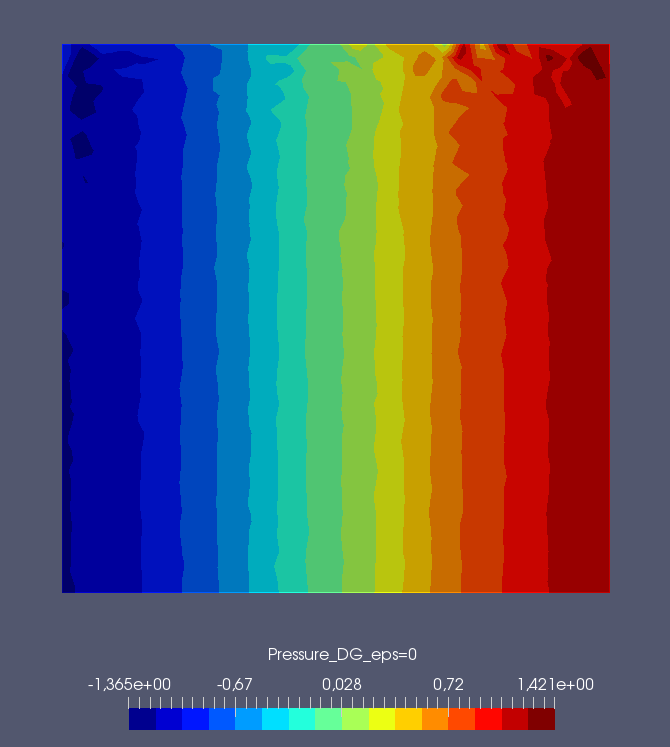}
      \\
      Velocity field.
      &
        Pressure iso-values.
    \end{tabular}
  \end{center}
  \caption{Cavity test, $\P1/\P1$ SIP DG}
  \label{fig:P1P1cavity}
\end{figure}

\begin{figure}
  \begin{center}
    \begin{tabular}{@{}cc@{}}
      \includegraphics[width=0.35\linewidth, height=0.37\linewidth]{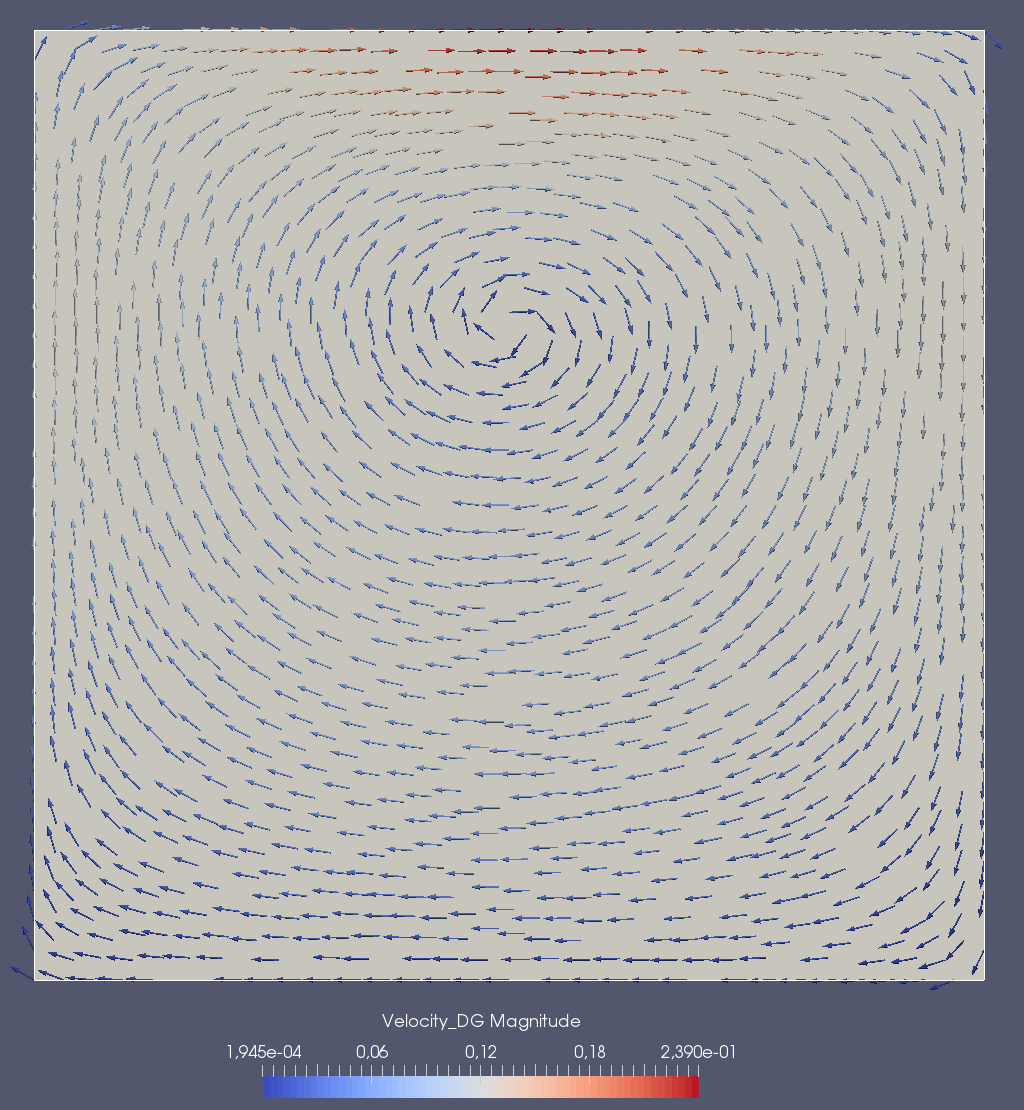}
      &
        \includegraphics[width=0.35\linewidth,height=0.37\linewidth]{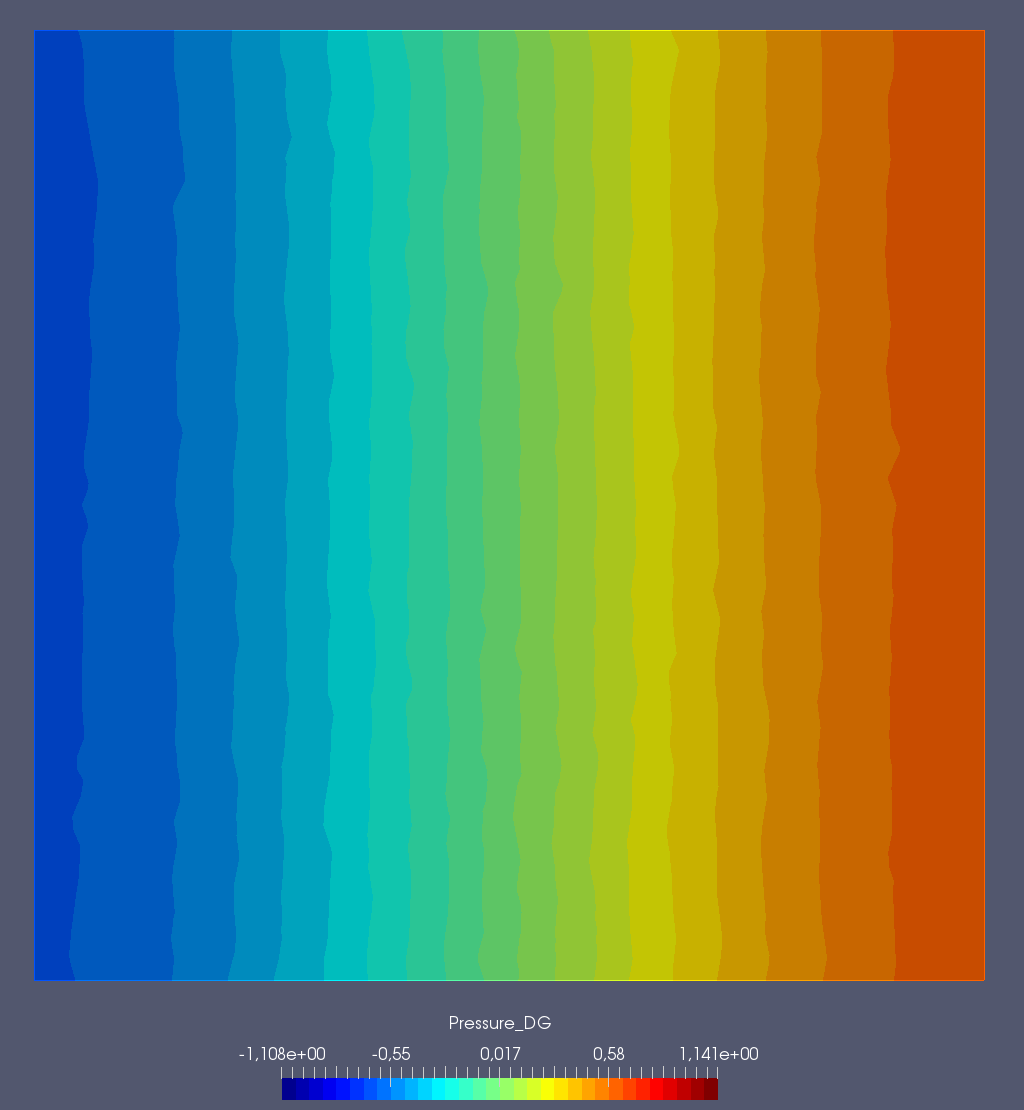}
      \\
      Velocity field.
      &
        Pressure iso-values.
    \end{tabular}
  \end{center}
  \caption{Cavity test, $\P2/\P2$ SIP DG}
  \label{fig:P2P2cavity}
\end{figure}
\section*{Acknowledgements }
The first author has been partially financed by the MINECO grant
MTM2015-69875-P (Spain) with the participation of FEDER.  The second
and third authors are also partially supported by the research group
FQM-315 of Junta de Andalucía.

%%%%%%%%%%%%%%%%%%%%%%%%%%%%%%%%%%%%%%%%%%%%%%%%%%%%%%%%%%%%%%%%%%%%%%

\bibliography{biblio}
\end{document}